\documentclass[12pt]{article}
\makeindex
\usepackage{amsfonts,amssymb,amsmath,amscd,amsthm,xcolor,mathtools,enumerate}

\usepackage{latexsym}
\usepackage{graphicx}
\usepackage[backref=false,colorlinks, linkcolor=blue, citecolor=red]{hyperref}

\newtheorem{theorem}{Theorem}[section]
\newtheorem{proposition}[theorem]{Proposition}  
\newtheorem{corollary}[theorem]{Corollary}  
 
\newtheorem{lemma}[theorem]{Lemma}

\numberwithin{equation}{section}

\title{ Exact bounds for the sum of the inverse-power of element orders in non-cyclic finite abelian groups} 
\author{M. Archita } 
\date{ } 

\begin{document} 
	
	\maketitle 
	
	\begin{abstract}Given a finite abelian group $G$ of order $n,$ and denote the sum of the inverse-power of the element orders in $G$ by $m(G).$ Let $\mathbb{Z}_n$ be the cyclic group of order $n.$ Suppose $G$ is a non-cyclic abelian group of order $n.$ Then we show that $m(G)\geq \frac{3+t}{2+t}m(\mathbb{Z}_n)$, where $n=2^t\cdot l$, $t\geq 0$ and $\gcd(2,l)=1.$ It improves the inequality $m(G)>m(\mathbb{Z}_n)$ obtained by Azad and et al. Moreover, the obtained bound in this article is the best, as for $n=2^tl,$ $t\geq0,$ $l$ odd, there exists a non-cyclic abelian group $G$ of order $n$ satisfying $m(G)=\frac{3+t}{2+t}m(\mathbb{Z}_n)$. Later, we will obtain a lower bound for the ratio $\frac{m(\mathbb{Z}_n)}{m(G)}$. So, we finally establish that that $\frac{1}{p-1}m(G)< m(\mathbb{Z}_n)\leq \frac{2+t}{3+t} m(G),$ where $p$ is the largest prime dividing $n.$ These results provide a more precise understanding of the behavior of
$m(G)$ relative to the cyclic case and establish optimal bounds in the context of finite abelian groups.
	\end{abstract}

	\section{Introduction} 
	The structure of a group has been studied by examining at the element orders using different invariants. For example, several authors have studied various structural properties of a finite group using the function $\psi(G)=\sum_{a\in G} o(a),$ where $o(a)$ denotes the order of $a\in G,$ for example see \cite{HMLM}. In 2017, Garonzi, Martino, and Massimiliano Patassini introduced the function $m(G)=\sum_{a\in G}\frac{1}{o(a)},$ in \cite{GMP}. Recently, in 2023, Morteza Baniasad Azad, Behrooz Khosravi and Hamideh Rashidi studied the group theoretical properties such as solvability, supersolvability and nilpotency of a finite group $G,$ using the function $m(G),$ ({{}cf. \cite[Theorem $2.7,$ $2.8,$ and $2.9$ ]{BKR}})
	\medskip
	
	In the following theorem, Garonzi, Martino, and Massimiliano Patassini provided an inequality to investigate the cyclic property of a finite group $G.$

	
	\begin{theorem}	[{{} cf. \cite[Theorem $5$]{GMP}}]\label{initial}
		\textit{For any finite group $G$ of order $n$, $m(G)\geq m(\mathbb{Z}_n)$ and equality holds if and only if $G\simeq\mathbb{Z}_n,$ where $\mathbb{Z}_n$ is the cyclic group of order $n.$}
	\end{theorem}
	
	In this paper, we continue our study of structural properties of a finite abelian group using $m(G).$ The following theorem is our main result, which gives an exact bound and strengthens the bound in Theorem \ref{initial}.
	\begin{theorem}\label {main theorem}
		Let $G$ be a non-cyclic abelian group of order $n,$ and $\mathbb{Z}_n$ is the cyclic group of order $n.$  Then $$m(G)\geq\frac{3+t}{2+t} m(\mathbb{Z}_n),$$ where $n=2^t\cdot l,$ $t\geq 0$, and $\gcd(2,l)=1.$
	\end{theorem}
	In the following proposition, it is shown that there exists a finite group of order $n$ such that 
	this lower bound is the best possible.
	\begin{proposition}\label{exact}
		Let $n=2^tl,$ where $t\geq0,$ and $l$ is an odd integer. Then $$m(\mathbb{Z}_n)=\frac{2+t}{2}m(\mathbb{Z}_l)  \hspace{8 mm} m(\mathbb{Z}_{2l}\times \mathbb{Z}_{2^{t-1}})=\frac{3+t}{2} m(\mathbb{Z}_l)$$ and hence $$m(\mathbb{Z}_{2l}\times \mathbb{Z}_{2^{t-1}})=\frac{3+t}{2+t}m(\mathbb{Z}_n).$$
	\end{proposition}
	In the view of Theorem \ref{main theorem} and Proposition \ref{exact}, it is shown that, for each $n=2^t\cdot l,$ where $t\geq0,$ $l$ is an odd, there exist a group of order $n$ satisfying $m(G)=\frac{3+t}{2+t}m(\mathbb{Z}_n).$
	\medskip
	
	In the next theorem we provides an improvement for non-cyclic finite abelian groups of an odd order.
	\begin{theorem}\label{lower bound}
		Let $G$ be a non-cyclic finite abelian group of an odd order $n,$ and $p$ be the largest prime dividing $n.$ Then $$m(\mathbb{Z}_n)>\frac{1}{(p-1)}m(G).$$
	\end{theorem}
	
	\section{Preliminaries}
	In this section we will describe some instrumental results which are already developed for the function $m(G).$
	\begin{lemma}[{{}cf. \cite[Lemma 2.2]{BKR}}]
		If $A\leq G$ then $m(A)\leq m(G)$ and equality holds if and only if $A=G.$
	\end{lemma}
	\begin{lemma}[{{}cf. \cite[Lemma 2.3]{BKR}}]
		If $A \trianglelefteq G$ then $m(G/A)\leq m(G)$ and equality holds if and only if $A=1.$
	\end{lemma}
	\begin{lemma}[{{}cf. \cite[Lemma 2.6]{BKR}}]\label{gcd}
		For any two finite groups $A_1$ and $A_2,$ we have $m(A_1\times A_2)\geq m(A_1)\cdot m(A_2).$ Moreover, equality holds if and only if $\gcd(|A_1|,|A_2|)=1.$
	\end{lemma}

	\begin{lemma}[{{}cf. \cite[Theorem 2.20]{I}}]\label{core}
		Let $H$ be a proper cyclic subgroup of a finite group $G,$ and let $M=core_{G}(H).$ Then $|H:M|<|G:H|,$ and in particular, if $|H|>|G:H|,$ then $M>1.$
	\end{lemma}

	\section{Preparatory Bounds}	
	In the following lemmas we establish some bounds for $m(\mathbb{Z}_n).$
	\begin{lemma}\label {2}
		Let $2\leq q=p_1<p_2<\cdots < p_s=p$ be the prime divisors of $n$ and the corresponding Sylow subgroups of $\mathbb{Z}_n$ are $P_1, P_2,\cdots P_s.$ Let the highest power of each $p_i$ in $P_i$ i.e., exponent of $p_i$ in $(P_i)$ are $r_{i}.$
		 Then $$m(\mathbb{Z}_n)<\frac{2+t}{3+t}\sqrt{n},$$ where $n=2^t\cdot l,$ $t\geq 0$, and $\gcd(2,l)=1.$
	\end{lemma}
	\begin{proof}
		Let $n=\prod_{i=1}^s(p_{i}^{r_i}).$ Lemma \ref {gcd} stipulates	$m(\mathbb{Z}_n)=\prod_{i=1}^s(m(P_i)),$ and for each $i,$ we have $m(P_i) =(\frac{p_i-1}{p_i}r_i+1).$  
		\medskip 
		
		So, $m(\mathbb{Z}_n)=\prod_{i=1}^{s}(\frac{p_i-1}{p_i}r_i+1)<\prod_{i=1}^{s}( r_i+1).$ For each natural number $r\geq 1,$ we have $r+1\leq2^{r}=4^{r/2}< 
		 (\frac{2+t}{3+t})^{r/2}(p)^{r/2},$ for all $p\geq 7.$ We will complete the rest of the proof dividing in the following cases.
		\newline
		
		\underline{\it{case (a):} $p_1=2,$ $r_1\neq 1$ or $r_1=0,$ and $r_i\geq 0$ for all $2\leq i \leq s$} 
		\medskip
		
		Note that, $(\frac{r}{2}+1)\leq 2^{r/2}$ for $r=0,$ and for all natural number $r\geq 2.$ Also, $(\frac{2}{3}r+1)<3^{r/2}$ for all $r\geq 1$, and $(\frac{4}{5}r+1)<5^{r/2}$ for all $r\geq 1$. Now, $m(\mathbb{Z}_n)=(
		\frac{r_1}{2}+1)(\frac{2r_2}{3}+1)(\frac{4r_3}{5}+1)\prod_{i=4}^{s}(\frac{p_i-1}{p_i}r_i+1)<2^{r_1/2}3^{r_2/2}5^{r_3/2}\prod_{i=4}^{s}(\frac{2+t}{3+t})^{r_i/2}(p_i)^{r_i/2}.$ Hence, $m(\mathbb{Z}_n)<\frac{2+t}{3+t}\sqrt{n}.$\\

			\underline{\it{case (b):} $p_1=2,$ $r_1= 1$ and $r_i\geq 0$ for all $2\leq i \leq s$} 
		\medskip
		
		In this case, $n=2\cdot l,$ where $l$ is an odd integer. So, $m(\mathbb{Z}_n)=m(\mathbb{Z}_l)\cdot m(\mathbb{Z}_2)<\frac{2}{3}\sqrt{l}(1+\frac{1}{2})=\sqrt{l}=\frac{\sqrt{n}}{\sqrt{2}}<\frac{3}{4}\sqrt{n}.$
		\medskip
		
		  Hence, we obtain that $m(\mathbb{Z}_n)<\frac{2+t}{3+t}\sqrt{n},$ where $n=2^t\cdot l,$ $t\geq 0$,   and $\gcd(2,l)=1.$
		\medskip
		

	\end{proof}
	Next, in the following lemma we calculate some lower bounds for $m(\mathbb{Z}_n).$
	\begin{lemma}\label {basic bound}
		Let $2\leq q=p_1<p_2<\cdots < p_t=p$ be the prime divisors of $n$ and the corresponding Sylow subgroups of $\mathbb{Z}_n$ are $P_1, P_2,\cdots P_t.$ Then $$\frac{1}{p}<m(\mathbb{Z}_n),\hspace{8 mm} \hbox{and}\hspace{8 mm} m(\mathbb{Z}_n)>\frac{\phi(n)}{n},$$ where $\phi(n)$ is the Euler's phi-function.
	\end{lemma}
	\begin{proof}
		Let $n=\prod_{i=1}^t(p_{i}^{r_i}).$ Lemma \ref {gcd} gives	$m(\mathbb{Z}_n)=\prod_{i=1}^t(m(P_i)),$ and for each $i,$ we have $m(P_i) =(\frac{p_i-1}{p_i}r_i+1).$   Now for each $i,$ it is clear that $\frac{p_i-1}{p_i}r_i+1>\frac{p_i-1}{p_i}$ and $p_{i+1}\geq p_i+1.$ Hence it follows that $m(\mathbb{Z}_n)>\frac{1}{p}.$
		\medskip

		For the another inequality note that, $m(\mathbb{Z}_n)=\prod_{i=1}^{t}(\frac{p_i-1}{p_i}r_i+1).$ The value of Euler's $\phi$-function is $\phi(n)=\prod_{i=1}^t\phi(p_i^{r_i})=\prod_{i=1}^t(\frac{p_i-1}{p_i})p_i^{r_i}.$ So, we use the fact $\frac{p_i-1}{p_i}<(\frac{p_i-1}{p_i}\cdot r_i+1),$ for each $i,$ and it follows that $\frac{\phi(n)}{n}<m(\mathbb{Z}_n).$
	\end{proof}
	To obtain another lower bound for $m(\mathbb{Z}_n),$ we will obtain a lower bound for the Euler's phi-function, $\phi(n),$ in the following lemma: 
	\begin{lemma}\label {bound for phi}
		Let $p$ be the largest prime divisor and $q$ be the smallest prime divisor of an integer $n>1$ respectively. 	Let $2\leq q=p_1<p_2<\cdots < p_t=p$ be the prime divisors of $n.$ Then $\frac{\phi(n)}{n}\geq \frac{q-1}{p}.$
	\end{lemma}
	\begin{proof}
		Let $t$ be the number of distinct prime divisors of $n.$ If $t=1,$ we have $n=p^r,$ for some $r\geq 1,$ and  $\phi(n)=(p-1)p^{r-1}=(p-1)(n/p).$ Next assume that $t>1,$ and the proof will be done by induction on $t.$
		\medskip
		
		Now $n=p^s\cdot m,$ where, $(p,m)=1.$ Then $m$ has exactly $t-1$ prime divisors. Hence, $$\phi(n)=\phi(p^s)\cdot \phi(m)\geq  (p-1)\frac{p^s}{p}\cdot \frac{m(q-1)}{p_{t-1}}=\frac{(p-1)(q-1)n}{pp_{t-1}}\geq \frac{n(q-1)}{p}.$$
		The last inequality holds since, $p-1\geq p_{t-1}.$
	\end{proof}
	\begin{lemma}\label{min}
		Let $p$ be the largest prime divisor and $q$ be the smallest prime divisor of an integer $n>1$ respectively. Then $m(\mathbb{Z}_n)\geq \frac{(q-1)}{p}.$	
	\end{lemma}
	\begin{proof}
		Combining the results of Lemma \ref{basic bound} and Lemma \ref{bound for phi}, the result follows.
	\end{proof}
	
		
	\section{Proofs of the main results} 
	We will first proof Theorem \ref{main theorem}. Later, we will prove Theorem \ref{lower bound}, which specially applies to finite groups of odd order.
	

	\begin{proof}{\it proof of Theorem \ref{main theorem}}\\ Equivalent statement of theorem is: If $m(G)<\frac{3+t}{2+t}m(\mathbb{Z}_n)$ then $G$ is a cyclic group. We will proof it by induction on the cardinality of the finite group $G.$\\
		If $|G|=1,$ then the result is trivial. So move to the case when $|G|>1.$ Suppose that $G$ contains a non-trivial normal cyclic subgroup $N$ such that $(|N|,|G/N|)=1.$	Then $m(G/N)<m(G).$	Now, Lemma \ref{gcd}, and Theorem \ref{initial} gives that,
		\begin{eqnarray}
			m(N)\cdot m(G/N)= m(N\times G/N)\leq m(G)<\frac{3+t}{2+t}m(\mathbb{Z}_n)=\frac{3+t}{2+t}m(\mathbb{Z}_{|N|})\cdot m(\mathbb{Z}_{|G/N|}).\nonumber
		\end{eqnarray}
		
		Since, $N$ is a cyclic group, $N\cong \mathbb{Z}_{|N|}.$ So we obtain that $m(G/N)<\frac{3+t}{2+t}m(\mathbb{Z}_{|G/N|}).$	This implies $G/N$ is a cyclic subgroup of $G.$ Hence, $G$ is a cyclic group.
		\medskip
		
		So we may assume that $G$ has no non-trivial normal cyclic subgroup with $(|N|,|G/N|)=1.$ Now using Lemma \ref{2}, we have that $$m(G)<\frac{3+t}{2+t}\cdot m(\mathbb{Z}_n)<\sqrt{n}.$$ So, there exists an element $x\in G$ such that $o(x)>\sqrt{n},$
		and	$|\langle x\rangle|>|G:\langle x\rangle|.$
		\medskip
		
		Using  Lemma \ref{core}, one can obtain $N=$ core$_{G}(\langle x \rangle)$ is a non-trivial normal cyclic subgroup of $G.$\label{same}
		\medskip  
		
		
		Suppose $(|N|,|G/N|)\neq 1.$ Hence, $G$ is an abelian group but not a cyclic group. Thus, $G$ has no element of of order $n.$ The least possible order of a non-identity element in $G$ is $q.$ Then $$m(G)\leq \frac{n-1}{q}+1.$$ In $\mathbb{Z}_n,$ the maximum possible value of a non-identity, non-generating element is $\frac{n}{q}.$ So, $$\frac{\phi(n)}{n}+\frac{(n-\phi(n))q}{n}<1+\frac{\phi(n)}{n}+\frac{(n-\phi(n)-1)q}{n}\leq m(\mathbb{Z}_n).$$ Theorem \ref{initial} and our assumption gives that, $$ m(\mathbb{Z}_n)\leq m(G)<\frac{3+t}{2+t}m(\mathbb{Z}_n)\leq \frac{3+t}{2+t}m(G).$$ So we obtain that, 
		\begin{eqnarray*}
			&\frac{\phi(n)}{n}+\frac{(n-\phi(n))q}{n}< \frac{3+t}{2+t}\cdot\frac{n-1}{q}+\frac{3+t}{2+t}\\
			&\Rightarrow \frac{\phi(n)}{n}\cdot (1-q)+q<\frac{3+t}{2+t}\cdot \frac{(n-1+q)}{q}\\
			& \Rightarrow \frac{\phi(n)}{n}\cdot (1-q)<\frac{((3+t)n-(3+t)+(3+t)q-(2+t)q^2)}{(2+t)q}\\
			& \Rightarrow \frac{\phi(n)}{n}\cdot \frac{(2+t)q}{((3+t)n-(3+t)+(3+t)q-(2+t)q^2)}<\frac{1}{1-q}<0\\
			& \Rightarrow (3+t)n-(3+t)+(3+t)q-(2+t)q^2<0\\
			& \Rightarrow (3+t)n<(3+t)-(3+t)q+(2+t)q^2<(3+t)+(3+t)q^2\\
			& \Rightarrow n<1+q^2.
		\end{eqnarray*}
		So it is a contradiction except for $n=q^2.$ Hence, except for the case $n=q^2$ we obtain $(|N|,|G/N|)= 1,$ and $G$ is a cyclic group of order $n.$ To complete the proof, we need to check the case where $n=q^2.$ Note that
		 $m(\mathbb{Z}_{q^2})=1+\frac{2(q-1)}{q},$ and $m(\mathbb{Z}_q\times \mathbb{Z}_q)=(1+\frac{q^2-1}{q}).$ As, $q\geq 2,$ it is obvious that $m(\mathbb{Z}_{q^2})\leq \frac{2+t}{3+t} m(\mathbb{Z}_{q}\times \mathbb{Z}_{q}).$ This completes the proof of Theorem \ref{main theorem}.
	\end{proof}

	We continue with the proof of Theorem \ref{lower bound}:
	\begin{proof}[proof of Theorem \ref{lower bound}]
	Suppose, $\frac{m(\mathbb{Z}_n)}{m(G)}\leq \frac{1}{p-1}$. Now, Theorem \ref{main theorem} exhibits the best upper bound for $\frac{m(\mathbb{Z}_n)}{m(G)}$ is $\frac{2+t}{3+t}.$ Hence, $\frac{1}{p-1}\geq \frac{2+t}{3+t}$, which implies that $p\leq 2+\frac{1}{2+t}.$ This is a contradiction expect the case $n=p^2.$ 
		\medskip
		
		For $n=p^2$, clearly one have $$\frac{1}{(p-1)}\cdot m(\mathbb{Z}_p\times \mathbb{Z}_p)=\frac{1}{p-1}+\frac{(p+1)}{p}\leq m(\mathbb{Z}_{p^2}),$$ as $p\geq 3.$ This completes the proof of the result.
	\end{proof}
	\begin{corollary}
		Let $G$ be a non-cyclic finite abelian group of order $n,$ where $n=2^t\cdot l$ with $\gcd(2,l)=1,$ and $p$ is the largest prime dividing $n.$ Let $\mathbb{Z}_n$ be the cyclic group of order $n.$ Then $$\frac{1}{(p-1)}<\frac{m(\mathbb{Z}_n)}{m(G)}\leq \frac{2+t}{3+t}.$$ 
	\end{corollary}
	\begin{proof}
		This corollary follows from  Theorem \ref{main theorem} and Theorem \ref{lower bound}.
	\end{proof}
	
		\hspace{50 mm}{\bf Acknowledgments}
		\newline
		
		The author is supported by a post-doctoral fellowship provided by the University of Antwerp,Belgium. 
        She 
        would like to thank the 
        Department of Mathematics, University of Antwerp for providing a wonderful research atmosphere and resources.

		\bigskip 
		
		\noindent M. Archita\\
		E-mail: {\ttfamily architamondal40@gmail.com},\\
		Department of Mathematics,\\
        University of Antwerp,\\
        Antwerp, Belgium.
		\vspace{.5cm}



\begin{thebibliography}{[XXX]} 
			
			
			
			\bibitem[BK]{BK} Azad, Morteza Baniasad, and Behrooz Khosravi. \emph{Properties of finite groups determined by the product of their element orders}, Bulletin of the Australian Mathematical Society 103.1 (2021): 88-95.
			
			
			\bibitem[BKR]{BKR} Azad, Morteza Baniasad, Behrooz Khosravi, and Hamideh Rashidi. \emph{On the sum of the inverses of the element orders in finite groups}, Communications in Algebra 51.2 (2023): 694-698.
			
			\bibitem[GMP]{GMP}Garonzi, Martino, and Massimiliano Patassini. \emph{Inequalities detecting structural properties of a finite group}, Communications in Algebra 45.2 (2017): 677-687.
			
			\bibitem[I]{I} Isaacs, I. Martin. Finite group theory. Vol. 92. American Mathematical Soc., 2008.
			
			\bibitem[HMLM]{HMLM} Herzog, Marcel, Patrizia Longobardi, and Mercede Maj. \emph{The second maximal groups with respect to the sum of element orders.} Journal of Pure and Applied Algebra 225.3 (2021): 106531.
			
			
		\end{thebibliography}
	\end{document}